\documentclass[11pt,reqno]{amsart}

\usepackage{amsmath}
\usepackage{amsfonts}
\usepackage{amssymb}
\usepackage{amsxtra}
\usepackage{amsthm}
\usepackage{amsmath,amssymb,amsfonts,amsthm,color,latexsym}
\usepackage{xcolor}
\usepackage{float}

\theoremstyle{plain}

\newcommand{\C}{\mathbb{C}}

\newcommand{\f}{\mathcal{F}}
\newcommand{\Le}{L}

\newcommand{\p}{\mathbb{P}}

\newcommand{\sing}{\mathrm{Sing}}

\newtheorem{maintheorem}{Theorem}

\newtheorem{secondtheorem}{Theorem}

\newtheorem{theorem}{Theorem}[section]
\newtheorem{lemma}[theorem]{Lemma}
\newtheorem{proposition}[theorem]{Proposition}
\newtheorem{corollary}[theorem]{Corollary}

\newtheorem{example}[theorem]{Example}
\newtheorem{remark}[theorem]{Remark}

\usepackage{graphicx}

\usepackage[pagebackref]{hyperref}

\begin{document}

\title[Codimension one foliations of degree four]{Holomorphic foliations of degree four on the complex projective space}

\date{\today}
\author[A. Fern\'andez-P\'erez]{Arturo Fern\'andez-P\'erez}
\address[Arturo Fern\'{a}ndez P\'erez] {Department of Mathematics. Federal University of Minas Gerais. Av. Ant\^onio Carlos, 6627 
CEP 31270-901\\
Pampulha - Belo Horizonte - Brazil\\}
\email{fernandez@ufmg.br}

\author[V\^angellis Sagnori Maia]{V\^angellis Sagnori Maia}
\address[V\^angellis Sagnori Maia] {Department of Mathematics. Federal University of Minas Gerais. Av. Ant\^onio Carlos, 6627 
CEP 31270-901\\
Pampulha - Belo Horizonte - Brazil\\}
\email{vangellis07@gmail.com}

\thanks{The first author is partially supported by CNPq-Brazil PDE 2019, Grant Number 201009/2020-0, and FAPEMIG  RED-00133-21, and Pronex-Faperj E-26/010.001270/2016. The second-named was supported by CAPES-Brazil.
}

\begin{abstract}
In this paper, we study holomorphic foliations of degree four on complex projective space $\p^n$, where $n\geq 3$, with a special focus on obtaining a structural theorem for these foliations. Furthermore, for a foliation $\f$ of degree $d\geq 4$ with a sufficiently high $k^{th}$-jet, we prove that either $\f$ is transversely affine outside a compact hypersurface, or $\f$ is transversely projective outside a compact hypersurface, or $\f$ is the pull-back of a foliation on $\p^2$ by a rational map.
\end{abstract}

\maketitle
\section{Introduction and Statement of the results}
The study of holomorphic foliations has attracted the attention of many mathematicians and has been gaining more prominence in recent decades. For example, it is well established that the
 theory of foliations plays in important role in the analysis of subvarieties within projective varieties. A compelling illustration of this fact is provided by Bogomolov's paper \cite{Bogo}, which concerns the renowned Green-Griffiths-Lang conjecture. Techniques from
 Algebraic Geometry have proven exceptionally valuable when exploring singular holomorphic foliations. Notably, in his acclaimed Lecture Notes \cite{jouanolou}, J.-P. Jouanolou established that a generic polynomial vector field of degree greater than one on the complex projective plane does not possess any invariant algebraic curve. Recently, a focal point in the theory of foliations has been the \textit{classification problem}, particularly for \textit{codimension one holomorphic foliations on $\p^n$}. 
 For instance, the following conjecture (cf. \cite[page 2]{cerveaualcidesgrau3}) is attributed to different authors such as Marco Brunella, Alcides Lins Neto, Dominique Cerveau, and others:\\
\indent \textbf{Conjecture 1.}
Any codimension one holomorphic foliation $\f$ on $\p^n$, with $n\geq 3$,
\begin{itemize}
  \item[(*)] either $\f$ is transversely projective outside a compact hypersurface; 
  \item[(**)] or $\f$ is the pull-back of a holomorphic foliation $\mathcal{G}$ on $\p^2$ by a rational map $\Phi:\p^n\dashrightarrow\p^2$.
\end{itemize}
\par The concepts used in the previous conjecture will be explained throughout this paper. We emphasize that every codimension one holomorphic foliation F on $\mathbb{P}^n$ corresponds to a homogeneous 1-form $\omega$ on $\mathbb{C}^{n+1}$ of degree $d + 1$ that defines $\pi^{*}(\f)$, whose singular set has codimension at least two. Here, $\pi : \mathbb{C}^{n+1} \setminus {0} \rightarrow \mathbb{P}^n$ represents the natural projection, and $\pi^{*}(\f)$ signifies the pull-back of $\f$ via $\pi$. By definition, the integer $d$ is the degree of the foliation $\f$. We remark that the integer $d$ is precisely the number of tangencies of $F$ with a generic line $L$. Consequently, we are equipped to explore the space of codimension one holomorphic foliations on $\mathbb{P}^n$ with a specific degree. The Zariski closure of this set is identified as an algebraic set and naturally has irreducible components
\par With a focus on explaining these irreducible components to a certain degree, significant progress has been made through several works. The initial case that can be considered involves the space of codimension one foliations on $\p^n$ of degree zero. It has been established that this space possesses only one component, isomorphic to the Grassmannian of lines in $\p^n$. A proof of this assertion can be found in \cite{deserticerveau}. In $\p^n$ with $n \geq 3$, the space of holomorphic foliations of degree one contains two irreducible components, as established by Jouanolou in \cite{jouanolou}. Later, Dominique Cerveau and Alcides Lins Neto resumed the study of irreducible components and presented a significant result that brought this subject back into the spotlight. In \cite{cervalccompdgretwo}, they proved that in $\p^n$, with $n \geq 3$, the space of codimension one foliations of degree two has six irreducible components. Furthermore, they explicitly illustrate the generic element of each one of these irreducible components. These components are called \textit{Linear pull-back foliations, Rational components, Logarithmic components, and an Exceptional component}. The paper of Cerveau-Lins Neto was an invaluable contribution and served as  motivation for numerous researchers to focus their studies on the classification of irreducible components of the space of holomorphic foliations on $\p^n$. Recently, in \cite{Correa}, Corr\^ea-Muniz proved that the space of foliations of degree two on $\p^n$ and dimension $k\geq 2$ also encompasses six irreducible components. 
\par Despite the progress made in these studies, solving the classification problem of irreducible components of the space of holomorphic foliations is far from simple. For example, even after a complete description of the irreducible components of the degree two foliations space on $\p^n$, with $n \geq 3$, it has not yet been possible to explain all the irreducible components of the space of codimension one holomorphic foliations on $\p^n$, $n\geq 3$, of degree $d\geq 3$. Nevertheless, Cerveau and Lins Neto presented a structural theorem for degree three foliations in \cite{cerveaualcidesgrau3}. More precisely, they proved that if $\f$ is a holomorphic codimension one foliation of degree three on $\p^n$, with $n\geq 3$, then
\begin{itemize}
    \item either $\f$ admits a rational first integral,
    \item or $\f$ is transversely affine outside a compact hypersurface,
    \item or $\f=\Phi^{*}(\mathcal{G})$, where $\Phi:\p^n\dashrightarrow\p^2$ is a rational map and $\mathcal{G}$ is a foliation on $\p^2$.
\end{itemize}
\par Recently, in \cite{extdgreethree}, F. Loray, J. V. Pereira and F. Touzet established a more refined structural theorem for codimension one foliations of degree three on $\p^3$. Building upon this, in \cite{extdgreethreepn}, R. C. da Costa, R. Lizarbe and J. V. Pereira extended the result to $\p^n$, $n \geq 3$. Furthermore, Da Costa, Lizarbe and Pereira \cite[Theorem B]{extdgreethreepn} provide a complete list of the irreducible components of the space of foliations of degree three on $\p^n$, $n\geq 3$, whose general elements do not admit a rational first integral. However, even in \cite{extdgreethreepn}, a complete classification of the irreducible components of the space of foliations of degree three on $\p^n$, $n\geq 3$ is not known. Nevertheless, they have established that the space of codimension one foliations of degree three on $\p^n$, $n\geq 3$, have at least 24 distinct irreducible components.
\par Inspired by the previous results, this paper is devoted to studying degree four foliations on $\p^n$, with $n\geq 3$. Our main result is the following:
\begin{maintheorem}\label{main_theo}
Let $\f$ be a codimension one holomorphic foliation of degree four on $\p^n$, with $n \geq 3$. Then, 
\begin{itemize}
    \item[(i)] either $\f$ admits a rational first integral;
    \item[(ii)] or $\f$ is transversely affine outside a compact hypersurface;
    \item[(iii)] or $\f$ is a pure transversely projective outside a compact hypersurface;
    \item[(iv)] or $\f = \Phi^{\ast}(\mathcal{G})$, where $\Phi: \p^n \dashrightarrow \p^2$ is a rational map and $\mathcal{G}$
    is a holomorphic foliation on $\p^2$.
    \item[(v)] or there exists a birational map $\Psi:\mathbb{P}^{n-1}\times\mathbb{P}^{1}\dashrightarrow \mathbb{P}^{n}$ such that the foliation $\Psi^{*}(\f)$ is defined by a 1-form described as follows:
\[\eta_t = \beta_0 + t\beta_1 + t^2 \beta_2  + t^3 \beta_3 + t^4 \beta_4 -tdt,\]
where the 1-forms $\beta_j$ do not depend on $t\in\mathbb{P}^{1}$, for all $0\leq j\leq 4$.
\end{itemize}
\label{teoprincip}
\end{maintheorem}
\par Note that in order to confirm Conjecture 1 for degree four codimension one foliations on $\p^n$, $n\geq 3$, it is necessary to prove that item $(v)$ is equivalent to some of the previous items. We hope to prove this fact in the near future. One of the difference between Theorem \ref{main_theo} and the structural theorem of degree three foliations given by Cerveau-Lins Neto \cite{cerveaualcidesgrau3} is that pure transversely projective foliations exist. That is, there are foliations of degree four with projective transverse structure that is not affine. See the example given in \cite[Section 5.4]{artigoteogvs}.
\par Now, let us focus on foliations of degree $d\geq 4$. Let $\f$ be a degree $d$ foliation on $\p^n$. Then, $\f$ can be represented in an affine coordinate system $\C^n \simeq E \subset \p^n$ by an integrable polynomial 1-form $$\omega_E =\displaystyle \sum_{j=0}^{d+1}\omega_j$$ where the coefficients of the 1-forms $\omega_j$ are polynomials homogeneous of degree $j$, $0 \leq j \leq d+1$, and $i_R(\omega_{d+1})=0$. Given $p \in E$, let $j_p^{k}(\omega_E)$ be the $k^{th}$-jet of $\omega_E$ at $p$, and let
$$\mathcal{J}(\f,p)=min\{k \geq 0: j_{p}^{k}(\omega_E) \neq 0\}.$$
Note that $\mathcal{J}(\f,p)$ depends only on $p$ and $\f$, and not on $E$ and $\omega_E$. Moreover, the singular set of $\f$ is given by 
\[\sing(\f)=\{p\in\p^n:\mathcal{J}(\f,p)\geq 1\}.\]
It is well-known that $\sing(\f)$ is an algebraic set and always contains irreducible components of codimension two, (cf. \cite{lins}). 
\par Motivated by the family of foliations of \cite[Section 5.4]{artigoteogvs}, where the first $k^{th}$-jets ($k<3$) of the 1-form defining this family are all zero, we propose the following structural theorem for foliations of degree $d\geq 4$ in $\p^n$, with $n\geq 3$.
\begin{secondtheorem}\label{teorema_B}
Let $\f$ be a codimension one holomorphic foliation of degree $d\geq 4$ on $\p^n$, with $n \geq 3$. Suppose that one of the two conditions is satisfied: 
 \begin{enumerate}
 \item for all $p\in\sing(\f)$, we have $\mathcal{J}(\f,p)=1$;
 \item there exists $p\in\sing(\f)$ such that $\mathcal{J}(\f,p)\geq d-1$.
 \end{enumerate}
Then,
\begin{itemize}
    \item[(i)] either $\f$ admits a rational first integral;
    \item[(ii)] or $\f$ is transversely affine outside a compact hypersurface;
    \item[(iii)] or $\f$ is a pure transversely projective outside a compact hypersurface;
    \item[(iv)] or $\f = \Phi^{\ast}(\mathcal{G})$, where $\Phi: \p^n \dashrightarrow \p^2$ is a rational map and $\mathcal{G}$
    is a holomorphic foliation on $\p^2$.
\end{itemize}
\end{secondtheorem}
\par To prove Theorems \ref{main_theo} and \ref{teorema_B}, we will use the tools of  \cite{cerveaualcidesgrau3} and techniques concerning foliations admitting a finite \textit{Godbillon-Vey} sequence (cf. \cite{Scardua} and \cite{artigoteogvs}). 
\par The paper is organized as follows: in Section \ref{seccion_2}, we introduce the concept of holomorphic foliations and present some important results about codimension one foliations on $\p^n$. We also define the notions of affine and projective transverse structures of a foliation. Moreover, we establish the notion of a Godbillon-Vey sequence and state crucial results of foliations that admit a Godbillon-Vey sequence with finite length. Section \ref{seccion_3} is devoted to the proof of Theorem \ref{main_theo}, which will be broken down into several lemmas. Our proof is given step to step. Each step is according to length of a Godbillon-Vey sequence adapted to a degree four foliation.  Finally, in section \ref{prova_B}, we will prove Theorem \ref{teorema_B} using some lemmas used in the proof of Theorem \ref{main_theo}.
\section{Background and knows results}\label{seccion_2}
\subsection{Codimension one holomorphic foliations} 
Let $M$ be a complex compact connected manifold of dimension $n\geq 2$. A codimension one singular foliation $\mathcal{F}$ on $M$ is given by a 
covering by open subsets $\{U_{i}\}_{i\in I}$ on $M$ and a collection of integrable holomorphic 1-forms $\omega_{i}$ on $U_{i}$
satisfying: 
\begin{enumerate}
 \item $\omega_{i}\wedge d\omega_{i}=0$.
\item The singular set $\sing(\omega_{i})=\{p\in U_{i}: \omega_{i}(p)=0\}$ of $\omega_{i}$  is of
codimension at least two.  
\item On each non-empty intersection $U_{i}\cap U_{j}$
\begin{align*}
\omega_{i}=g_{ij}\omega_{j},\,\,\,\,\, \text{with}\,\,\,\,\, g_{ij}\in\mathcal{O}^{*}(U_{i}\cap U_{j}).
\end{align*}
\end{enumerate}
 The condition (3) implies that the singular set of $\mathcal{\mathcal{F}}$ defined by $\sing(\mathcal{F}):=\cup_{i\in I}\sing(\omega_{i})$ is a
 complex variety of $M$ of codimension at least two.
\par In the special case where $M$ is a projective manifold and $\mathcal{F}$ a foliation as above, we can associate to $\mathcal{F}$ a
meromorphic 1-form $\omega$ in the following way. We take a rational vector field $Z$ on $M$, not tangent to $\mathcal{F}$, that is $h_j=i_{Z_{|U_j}}\omega\not\equiv 0$; the meromorphic 1-form $\omega$ defined on $U_{j}$ by $\omega_{|U_{j}}=\omega_{j}/h_{j}$ is global and integrable. In this
case, we will say that $\omega$ defines $\mathcal{F}$.
\par We specialize to the case $M=\mathbb{P}^{n}$, the $n$ dimensional complex projective space. In that context, we can define $\mathcal{F}$ as follows: let $\pi:\mathbb{C}^{n+1}\setminus\{0\}\rightarrow\mathbb{P}^{n}$ be the natural projection, 
and consider $\pi^{*}(\mathcal{F})$ the pull-back of $\mathcal{F}$ by $\pi$; with the previous notations, $\pi^{*}(\mathcal{F})$ is defined by the 1-form
$\pi^{*}(\omega_j)$ on $\pi^{-1}(U_j)$. Recall that, for $n\geq 2$, we have $H^{1}(\mathbb{C}^{n+1}\backslash\{0\},\mathcal{O}^{*})=\{1\}$ (cf. \cite{cartan}). Consequently, there exists a global holomorphic 1-form $\omega$ on $\mathbb{C}^{n+1}\backslash\{0\}$ which defines $\pi^{*}(\mathcal{F})$ on $\mathbb{C}^{n+1}\backslash\{0\}$. By Hartog's prolongation theorem $\omega$ can be extended holomorphically at $0$. By construction we have $i_{R}\omega=0$, where $R$ is the Euler (or radial) vector field:
$$R=\sum^{n}_{j=0}z_{j}\partial_{z_{j}}.$$
 This fact and the integrability condition imply that $\omega$ is co-linear to an integrable homogeneous 1-form $\sum_{j=0}^{n}A_{j}(z)dz_{j}$, where the $A_{j}$'s are homogeneous polynomials of degree $d+1$, $gcd(A_0,\ldots,A_n)=1$. Therefore, we can state that every codimension one holomorphic foliation $\mathcal{F}$ on $\mathbb{P}^{n}$ corresponds to a homogeneous integrable 1-form $\omega$ on $\mathbb{C}^{n+1}$ that defines $\pi^{*}(\mathcal{F})$ with $cod(\sing(\omega))\geq 2$.
By definition, the integer $d$ is the degree of the foliation $\mathcal{F}$. Observe that 1-form $\omega$ is well-defined up to multiplication by non-zero complex number.
\par The following result will be crucial to demonstrate our theorems.
\begin{corollary}[Cerveau - Lins Neto \cite{cerveaualcidesgrau3}]
Let $\f$ be a codimension one holomorphic foliation on $\p^n$, $n \geq 3$. If $\mathcal{J}(\f, p) \leq 1, \; \forall \; p \in \p^n$, then $\f$ has a rational first integral.
\label{cor1}
\end{corollary}
\par Let $\f$ be a codimension one holomorphic foliation on a projective manifold $M$. Assume that $\omega$ is a meromorphic 1-form defining $\f$. We say that $\f$ is \textit{transversely projective} or $\f$ has a \textit{transverse projective structure} (cf. \cite{TeseScardua}) when there are meromorphic 1-forms $\omega_0=\omega$, $\omega_1$ and $\omega_2$ on $M$ satisfying 
\[
\left\{\begin{array}{lll}
    d\omega_0 = \omega_0 \wedge \omega_1 \\
    d\omega_1 = \omega_0 \wedge \omega_2\\
    d\omega_2 = \omega_1 \wedge \omega_2
    \end{array}\right.\]
    This means that, outside the polar and singular set of the 1-forms $\omega_i$, the foliation $\f$ is regular and transversely projective in the classical sense (see \cite{godbillon}, or \cite[Section 2.2]{artigoteogvs}). When $\omega_2\neq 0$, we say $\f$ is \textit{pure transversely projective}. If $\omega_2=0$, i.e. $d\omega_1=0$, we say that $\f$ is \textit{transversely affine} or $\f$ has a \textit{transverse affine structure}. 
    \par To end this subsection, let us give some examples of foliations on $\mathbb{P}^n$.
\begin{example}[Foliations with rational first integral]
Let $P,Q:\mathbb{C}  ^{n+1}\rightarrow\mathbb{C}$ be a homogeneous polynomials such that $\deg(P)=\deg(Q)=k\geq1$. Then $\omega=QdP-PdQ$ defines a 
codimension one foliation $\mathcal{F}$ on $\mathbb{P} ^n$ with a rational first integral $P/Q$.
\end{example}
\begin{example}[Foliations associated to closed meromorphic 1-forms]
If $\omega$ is a closed meromorphic 1-form on $\mathbb{P}^n$, $n\geq 2$, then it defines a codimension one holomorphic foliation on $\mathbb{P}^n$. According to  \cite[Proposition. 1.2.5]{livro}, we have that $\omega$ has a decomposition 
$$\omega=\sum_i\lambda_{i}\frac{df_i}{f_i}+dh,$$
where the $\lambda_i$'s are complex numbers and the $f_i$'s and $h$ are rational functions. The leaves are (outside the singular set of the foliation) the connected components of the level sets of the multi-valued function $\sum_i\lambda_i\log f_i+h$.
\end{example}

\subsection{Godbillon-Vey sequences \cite{godbillon}} 
Let $\f$ be a codimension one holomorphic foliation on a projective manifold $M$.
A \textit{Godbillon-Vey sequence} for $\f$ (briefly G-V-S) is a sequence $(\omega_0,\omega_1,\ldots,\omega_k,\ldots)$ of meromorphic 1-forms on $M$ such that $\f$ is defined by $\omega_0=0$, and the formal 1-form 
\begin{equation}\label{form2}
    \Omega=dz+\omega_0+z\omega_1+\frac{z^2}{2}\omega_2+\sum_{k\geq 3}^{\infty}\frac{z^k}{k!}\omega_k,
\end{equation}
is integrable, that is, $\Omega\wedge d\Omega=0$. In this case, the 1-form in (\ref{form2}) is meromorphic and can be extended meromorphically to $M\times\p^1$. Since it is integrable, it defines a codimension one foliation $\mathcal{H}$ on $M\times\p^1$ such that $\mathcal{H}|_{M\times\{0\}}=\f$.
\par When a meromorphic vector field $X$ exists on $M$ which is transversal to $\f$ at a generic point, then a unique meromorphic 1-form $\omega$ defining $\f$ exists, satisfying $i_X(\omega)=1$. According to \cite[Section 2.1]{artigoteogvs}, we can define a Godbillon-Vey sequence for $\f$ by setting $\omega_k:=L_X^{k}(\omega)$, where $L_X^k(\omega)$ denotes the $k^{th}$ Lie derivative along $X$ of the 1-form $\omega$. 
\par When there exists $N\in\mathbb{N}$ such that $\omega_N\neq 0$ but $\omega_{k}=0$ for all $j>N$ then we say that $\f$ admits a \textit{finite G-V-S of length} $N$. In general, the length is infinite. If $\f$ admits a G-V-S of length $N\leq 2$ then $\f$ is transversely projective outside a compact hypersurface. When $N=1$ then $\f$ has a transverse affine structure, see for instance \cite{TeseScardua} and \cite{godbillon}. 
\par For foliations that admit G-V-S of length $\geq 3$, we have the following result:

\begin{theorem}[Cerveau - Lins-Neto - Loray - Pereira - Touzet \cite{artigoteogvs}]
Let $\f$ be a codimension one foliation on a complex manifold $M$ that admits a G-V-S of length $N \geq 3$. Then

\begin{itemize}
    \item either $\f$ is transversely affine;
    \item or there is a compact Riemann surface $S$, 1-meromorphic forms $\alpha_0, \cdots, \alpha_N$ in $S$ and a rational map $\phi :M \rightarrow S \times \p^1$ such that $\f$ is defined by the 1-form $\Omega = \phi^{\ast}(\eta)$, where
    \begin{equation}
        \eta = dz + \alpha_0 + z\alpha_1 + \cdots + z^{N}\alpha_N.
        \label{eqteogvs}
    \end{equation}
\end{itemize}
When $M= \p^n$, $n \geq 3$, necessarily $S = \p^1$ and the 1-form in \ref{eqteogvs} can be written as
$$\eta = dz - P(t,z)dt,$$
where $P \in \C(t)[z]$ and $\f = \phi^{\ast}(\mathcal{G})$, where $\mathcal{G}$ is defined in $\p^ 1 \times \p^1$ by the differential equation $\frac{dz}{dt} = P(t,z)$.
\label{teogvs}
\end{theorem}
\begin{remark}\label{pull-back-gvs}
Let $\f$ and $\mathcal{G}$ be codimension one foliations on complex manifolds $X$ and $Y$, respectively. Suppose that $\mathcal{G}$ admits a finite G-V-S of length $N$ and $\f=\Phi^{*}(\mathcal{G})$, where $\Phi:X\dashrightarrow Y$ is a rational map. Then $\f$ also admits a G-V-S of length $N$.
\end{remark}
\begin{remark}\label{teo-rami}
Let $\f$ and $\mathcal{G}$ be codimension one foliations on complex manifolds $X$ and $Y$, respectively. Suppose that $\f=\Phi^{*}(\mathcal{G})$, where $\Phi:X\dashrightarrow Y$ is a dominant meromorphic map. Then, $\mathcal{G}$ is transversely projective (resp. affine) if, and only if, so is $\f$, (cf. \cite[Theorem 2.21]{artigoteogvs}). 
In the case $\Phi$ is a finite ramified covering, the above statement is equivalent to Theorem 1.6 (resp. Theorem 1.4) in \cite{Casale}.
\end{remark}

\section{Proof of Theorem \ref{main_theo}}\label{seccion_3}

\par Let $\f$ be a codimension one holomorphic foliation of degree four in $\p^n$, $n \geq 3$. In order to prove Theorem \ref{main_theo}, we consider two possibilities:
\begin{enumerate}
    \item for all $p\in\sing(\f)$, we have $\mathcal{J}(\f,p)=1$,
    \item there exists $p\in\sing(\f)$ such that $\mathcal{J}(\f,p)\geq 2$. 
\end{enumerate}
 In the first case, $\f$ admits a rational first integral by invoking Corollary \ref{cor1}. This consequently establishes the validity of assertion $(i)$ within Theorem \ref{main_theo}.
\par Therefore, we shall assume that there exists a point $p \in \p^n$ such that $\mathcal{J}(\f,p) \geq 2$. 
By employing affine coordinates $(z_1,\ldots,z_n)\in \C^n \subset \p^n$, where $p=0 \in \C^n$, we can conveniently consider $\f|_{\C^n}:\omega=0$, where $\omega$ is a polynomial 1-form in $\C^n$ expressed as follows: 
\begin{equation}\label{eq_1}
\omega = \alpha_2 + \alpha_3 + \alpha_4 + \alpha_5,
\end{equation} 
here, $\alpha_j$ corresponds to  homogeneous polynomial 1-forms of degree $j$, $2\leq j\leq 5$, and 
\begin{eqnarray}\label{eq_2}
i_R (\alpha_5)=0,\quad \quad\text{with}\quad R = \sum_{i=1}^{n} z_i \partial z_i.
\end{eqnarray}
We shall express $\alpha_j$ as: $\displaystyle \alpha_j (z) := \sum_{i=1}^{n} P_{ji}(z)dz_i$, with $2\leq j\leq 5$. Additionally, we  introduce
\[F_j (z) :=i_{R}(\alpha_{j-1})=\displaystyle \sum_{i=1}^{n}z_{i}\cdot P_{j-1 i}(z),\quad 3\leq j\leq 6,\] where $P_{j-1 i}$ are homogeneous polynomials of degree $j-1$. Note that $F_6 \equiv 0$ by (\ref{eq_2}). 
\par We proceed to examine the pull-back of 
$\omega$ through the process of  blowing-up of $\p^n$ at $0 \in \C^n \subset \p^n$. Let $\sigma : \tilde{\p}^n \rightarrow \p^n$ denote the blow-up at $0 \in \C^n\subset\p^n$, and let $\tilde{\f}$ represent the strict transform of $\f$ by $\sigma$. Our objective is to calculate $\sigma^{\ast}( \omega)$ within the chart
\begin{align}\label{explosion}
(\tau_1, \ldots, \tau_{n-1},x)=(\tau,x) \in \C^{n-1}\times \C \mapsto (x \tau,x) = (z_1, \ldots,z_n) \in \C^n \subset \p^n. 
\end{align}
 We have 
\begin{eqnarray}
\sigma^{\ast}(\omega) &= &\sigma^{\ast}(\alpha_2) +\sigma^{\ast}(\alpha_3) + \sigma^{\ast}(\alpha_4) + \sigma^{\ast}(\alpha_5) \nonumber \\
& = & (x^3\theta_2+x^2F_3(\tau,1)dx)+(x^4\theta_3+x^3F_4(\tau,1)dx)+(x^5\theta_4+x^4F_5(\tau,1)dx)+\nonumber\\
& & (x^6\theta_5+x^5F_6(\tau,1)dx)\nonumber\\
& =  & x^2\left(x\theta_2 + x^2\theta_3 + x^3\theta_4 + x^4\theta_5 + (F_3(\tau,1) + xF_4(\tau,1) + x^2F_5(\tau,1)+x^3 F_6(\tau,1))dx\right)\nonumber
\end{eqnarray}
where $$\displaystyle \theta_j = \sum_{i=1}^{n-1}P_{ji}(\tau,1)d\tau_i,\quad\quad 2\leq j\leq 5$$ depends only on $\tau$. 
Utilizing the condition $F_6(\tau,1)\equiv 0$, we derive the 1-form $\eta$ as follows: 
\begin{equation}\label{eq_3}
\eta=x\theta_2 + x^2\theta_3 + x^3\theta_4 + x^4\theta_5 + (F_3(\tau,1) + xF_4(\tau,1) + x^2F_5(\tau,1))dx.   
\end{equation}
 This 1-form serves to define the foliation $\tilde{\f}$ in the chart $(\tau,x)$.
\par Given the aforementioned conditions, we are presented with the subsequent possibilities for $F_i$:
\begin{itemize}
    \item[(1)] $F_3 \equiv F_4 \equiv F_5 \equiv 0$;
    \item[(2)] $F_4\equiv F_5 \equiv 0$, and $F_3 \not\equiv 0$;
    \item[(3)] $F_3 \equiv F_4 \equiv 0$, and $F_5 \not\equiv 0$;
    \item[(4)] Possibilities solved in an analogous way:
        \subitem (a) $F_3 \equiv F_5 \equiv 0$, and $F_4 \not\equiv 0$;
        \subitem (b) $F_3 \equiv 0$, and $F_4 \not\equiv 0 \not\equiv F_5$;
    \item[(5)] $F_5 \equiv 0$, and $F_3 \not\equiv 0 \not\equiv F_4$;
    \item[(6)] $F_4 \equiv 0$, and $F_3 \not\equiv 0 \not\equiv F_5$;
    \item[(7)] $F_3 \not\equiv 0$, $F_4 \not\equiv 0$, and $F_5 \not\equiv 0$.
\end{itemize}
We will investigate these potential scenarios by means of the following lemmas. Our exploration will revolve around $\omega$ and $\f$, taking into account the provided conditions, unless specified otherwise.
\begin{lemma}[Case 1]\label{caso1}
 If  $F_3\equiv F_4\equiv F_5 \equiv 0$. Then $\f$ is the pull-back by a linear map of a foliation on $\p^{n-1}$.
\end{lemma}
\begin{proof}[Proof]
Since $i_R(\omega)(z)=i_R(\omega)(x\tau,x)=x^3F_3(\tau,1)+x^4F_4(\tau,1)+x^5F_5(\tau,1),$ 
the assumption $F_3 \equiv F_4 \equiv F_5\equiv 0$ implies that $i_R(\omega)=0$. As deduced from 
(\ref{eq_1}), it is evident that $\alpha_5 \not\equiv 0$, otherwise $\f$ would have degree $\leq 3$. 
On the other hand, the integrability of $\omega$ implies that
\begin{align}
\omega \wedge d\omega = 0 & \Rightarrow i_R(\omega \wedge d\omega) = 0 \nonumber \\
& \Rightarrow i_R(\omega) \wedge d\omega - \omega \wedge i_R( d\omega) =0. \nonumber
\end{align}
Since $i_R(\omega)=0$, we have
\begin{equation}
    \omega \wedge i_R( d\omega) =0.
    \label{eq1}
\end{equation}
Upon applying Cartan's magic formula, we get
\begin{eqnarray}\label{eq2}
    L_R (\omega) &=& di_R(\omega) + i_R(d \omega)\nonumber\\
    L_R (\omega) &= & i_R(d \omega)
\end{eqnarray}
and applying Euler's formula to $L_{R}(\omega)$ yields
\begin{eqnarray}
L_R (\omega) &=& L_R (\alpha_2 + \alpha_3 + \alpha_4 + \alpha_5) \nonumber \\
& =& L_R (\alpha_2) + L_R (\alpha_3) + L_R (\alpha_4) + L_R (\alpha_5) \nonumber \\
& =& 3\alpha_2 + 4\alpha_3 + 5\alpha_4 + 6\alpha_5 \nonumber.
\end{eqnarray}
It follows from (\ref{eq2}) that $i_R(d\omega)=3\alpha_2 + 4\alpha_3 + 5\alpha_4 + 6\alpha_5$ and from  (\ref{eq1}), we deduce
\begin{eqnarray}
0=\omega\wedge i_R(d\omega) & = &  4\alpha_2\wedge\alpha_3 + 5\alpha_2\wedge\alpha_4 + 6\alpha_2\wedge\alpha_5 +  \nonumber \\
& & 3\alpha_3\wedge\alpha_2 + 5\alpha_3\wedge\alpha_4 + 6\alpha_3\wedge\alpha_5 + \nonumber \\
 & & 3\alpha_4\wedge\alpha_2 + 4\alpha_4\wedge\alpha_3 + 6\alpha_4\wedge\alpha_5 + \nonumber\\
 & & 3\alpha_5\wedge\alpha_2 + 4\alpha_5\wedge\alpha_3 + 5\alpha_5\wedge\alpha_4\nonumber\\
& = & \alpha_2\wedge\alpha_3 + 2\alpha_2\wedge\alpha_4 + 3\alpha_2\wedge\alpha_5+\nonumber\\
& & \alpha_3\wedge\alpha_4 + 2\alpha_3\wedge\alpha_5 + \alpha_4\wedge\alpha_5 \nonumber
\end{eqnarray}
Since the coefficients of $\alpha_j$ are homogeneous polynomials of degree $j$, each of these exterior products generates a homogeneous 2-form polynomial of a degree different from each other, so none of them can be a combination of the others, we obtain
$$\alpha_2\wedge\alpha_3 = \alpha_2\wedge\alpha_4 = \alpha_2\wedge\alpha_5 = \alpha_3\wedge\alpha_4 = \alpha_3\wedge\alpha_5 = \alpha_4\wedge\alpha_5 =0.$$
 Since $\alpha_5\not\equiv 0$ and 
$ \alpha_2\wedge\alpha_5 =  \alpha_3\wedge\alpha_5 = \alpha_4\wedge\alpha_5 =0$, we have that
 there exists meromorphic functions $f_j$, $j=2, 3, 4$, such that $\alpha_j = f_j\alpha_5$. 
We assert that $f_j\equiv 0$ for all $j=2,3,4$. Indeed, suppose by contradiction that some $f_j\not\equiv 0$. Then
$$\omega = f_2 \alpha_5 + f_3 \alpha_5 + f_4 \alpha_5 + \alpha_5 = (f_2 + f_3 + f_4 + 1)\alpha_5,$$
we have that the coefficients of $(f_2 + f_3 + f_4 + 1)\alpha_5$ would be of a degree different than 5, an absurd. Hence, the statement is proven.
\par Consequently, all $f_j = 0$, and thus $\alpha_j = 0$, for all $j = 2, 3, 4$. In particular, we have $\omega = \alpha_5$. Since $\alpha_5$ is integrable, it defines a foliation of degree four, say $\f_{n-1}$, on $\p^{n-1}$. If we consider $\p^{n-1}$ as the set of lines through $0\in\C^n\subset\p^n$ and $\Phi:\C^{n}\setminus\{0\}\to\p^{n-1}$ the natural projection then $\f=\Phi^{*}(\f_{n-1})$. This finishes the proof of the lemma. 
\end{proof}


\begin{lemma}[Case 2]
 Suppose that $F_4\equiv F_5\equiv 0$, and $F_3 \not\equiv 0$. Then, either $\f$ is transversely affine, or $\f$ is the pull-back by a rational map of a foliation on $\p^2$.
\label{caso2}
\end{lemma}
\begin{proof}
Let $\displaystyle\beta_j = \frac{\theta_{j+1}}{F_3(\tau,1)}$, for $1\leq j\leq 4$. It follows from (\ref{eq_3}) that
\begin{eqnarray}
\eta & = & x\theta_2 + x^2\theta_3 + x^3\theta_4 + x^4 \theta_5 + F_3 (\tau,1)dx \nonumber \\
\tilde{\eta}&=&\frac{\eta}{F_3(\tau,1)}= x\beta_1 + x^2\beta_2 + x^3\beta_3 + x^4 \beta_4 + dx.
\label{eq27}
\end{eqnarray}
We have that $\tilde{\eta}$ defines $\tilde{\f}$. 
Note that $\beta_{j}$ only depends on $\tau$, for all $1\leq j\leq 4$, because $\theta_j$ and $F_3$ only depend on $\tau$. Then, $\tilde{\eta}$ admits the G-V-S
$$(\eta_0 = \tilde{\eta},\eta_1,\eta_2,\eta_3 ,\eta_4),$$ where $\tilde{\eta}(\partial_x) =1$, $\eta_j = L_{\partial_x}^{j}(\tilde{\eta})$, $1\leq j\leq 4$. Here 
$\Le_{\partial_x}^{j}(\tilde{\eta})$ denotes the $j^{th}$ Lie derivative along $\partial_x$ of the form $\tilde{\eta}$, see for instance \cite[Section 2.1]{artigoteogvs}.\\
Within this context, two subcases arise:\\
\noindent {\textbf{Subcase I.}}
($\beta_4 \equiv 0$ and $\beta_3 \not\equiv 0$) or ($\beta_4 \not\equiv 0$). In this subcase, $\f$ admits a finite G-V-S of length $\geq 3$. Therefore, either $\f$ is transversely affine, or $\f$ is a pull-back by a rational map of a foliation on $\p^1 \times \p^1$ by Theorem \ref{teogvs}. Since $\p^1 \times \p^1$ is birational to $\p^2$, we get that $\f$ is a pull-back of a foliation on $\p^2$ by a rational map.  \\
\noindent {\textbf{Subcase II.}}
$\beta_4\equiv\beta_3 \equiv 0$. In this situation, $\f$ admits a finite G-V-S of length $\leq 2$, where $$\tilde{\eta} = dx + x\beta_1 + x^2 \beta_2$$
Taking the change of coordinates $x=\frac{1}{w}$, we get $\tilde{\eta}= -\frac{\Omega}{w^2}$, where 
$\Omega= dw - \beta_2 - w\beta_1.$ 
We assert that $\beta_1$ is closed, that is, $d\beta_1 =0$. Indeed, from the integrability condition of $\Omega$, we have
\begin{eqnarray}
0 = \Omega \wedge d\Omega &=& (dw - \beta_2 - w\beta_1) \wedge ( - d\beta_2 - dw \wedge \beta_1 -wd\beta_1) \nonumber \\
& = & -dw \wedge d\beta_2 - wdw \wedge d\beta_1 + \beta_2 \wedge d\beta_2 - \beta_2 \wedge \beta_1 \wedge dw + \nonumber \\
& & w\beta_2 \wedge d\beta_1 + w \beta_1 \wedge d\beta_2 + w^2 \beta_1 \wedge d\beta_1
\label{lie_10}
\end{eqnarray}
Using $2^{th}$ Lie derivative $\Le_{\partial_w}^{2}(\Omega \wedge d\Omega) = 2 \beta_1 \wedge d\beta_1 =0 \Rightarrow w^2 \beta_1 \wedge d\beta_1 =0$. Again using Lie derivative $L_{\partial_{w}}$ in (\ref{lie_10}), we get
\begin{align}
\Le_{\partial_w}(\Omega \wedge d\Omega) = -dw \wedge d\beta_1 + \beta_2 \wedge d\beta_1 + \beta_1 \wedge d\beta_2 =0.
\label{eq11}
\end{align}
From (\ref{eq11}), we can deduce that $dw\wedge d\beta_1=0$, as
$\beta_2 \wedge d\beta_1 + \beta_1 \wedge d\beta_2$ remains unaffected by $dw$. Given that $dw\neq 0$, and $d\beta_1$ is independent of $w$, we can conclude that $d\beta_1 =0$, thus confirming the statement. Upon differentiating $\Omega$, we have
\begin{align}\label{eq8}
d\Omega =& - d\beta_2 - dw\wedge\beta_1
\end{align} and thus
\begin{align}\label{eq9}
\beta_1 \wedge \Omega = \beta_1 \wedge dw - \beta_1 \wedge \beta_2.
\end{align}
\par We will now prove that equations $(\ref{eq8})$ and $(\ref{eq9})$ are equivalent. This equivalence establishes that $\Omega$ has a transversely affine structure, consequently implying the same for  $\f$. 
Proving this equality requires demonstrating that $$d\beta_2 = \beta_1 \wedge \beta_2.$$
It follows from (\ref{eq11}) that  $\beta_2 \wedge d\beta_1 + \beta_1 \wedge d\beta_2 =0 \Rightarrow w\beta_2 \wedge d\beta_1 + w\beta_1 \wedge d\beta_2 =0$. Then we summarize $\Omega \wedge d\Omega$ in
$$\Omega \wedge d\Omega = -dw\wedge d\beta_2 + \beta_2 \wedge d\beta_2 - \beta_2 \wedge \beta_1 \wedge dw =0$$
of the above equal $\beta_2 \wedge d\beta_2 =0$, while the other terms are dependent on $dw$. Hence
\begin{align*}
(-d\beta_2 - \beta_2 \wedge \beta_1)\wedge dw =0 \Rightarrow  d\beta_2 + \beta_2 \wedge \beta_1 =0 \Rightarrow d\beta_2 = \beta_1 \wedge \beta_2.
\end{align*}
As a consequence of equations (\ref{eq8}) and (\ref{eq9}), we obtain
$$d\Omega=\beta_1\wedge \Omega.$$
This implies that $\tilde{\eta}$ has a transversely affine structure, consequently establishing the same for $\f$.  
\end{proof}

\begin{lemma}[Case 3]\label{caso3}
 Suppose that $F_4\equiv F_3 \equiv 0$, and $F_5 \not\equiv 0$. Then either $\f$ is transversely affine, or $\f$ is the pull-back by a rational map of a foliation on $\p^2$, or $\f$ is pure transversely projective.
\end{lemma}
\begin{proof}
It follows from (\ref{eq_3}) that
\begin{eqnarray}\label{eq_20}
    \eta & = & x\theta_2 + x^{2} \theta_3 + x^3\theta_4 +x^4\theta_5 + x^{2}F_5(\tau,1)dx \\
    \tilde{\eta} & = & \frac{\eta}{x F_5(\tau,1)}= \frac{\theta_2}{F_5(\tau,1)} + \frac{x \theta_3}{F_5(\tau,1)}  + \frac{x^2\theta_4}{F_5(\tau,1)} + \frac{x^3\theta_5}{F_5(\tau,1)}  + xdx \nonumber
\end{eqnarray}
When we consider the birational map $\psi (\tau, z) = (\tau, \frac{1}{z})=(\tau,x)$, we obtain
\begin{eqnarray}
     \psi^{\ast}(\tilde{\eta})&=& \Bigg[\frac{\theta_2}{F_5} + \frac{ \theta_3}{zF_5}  + \frac{\theta_4}{z^2 F_5} + \frac{\theta_5}{z^3 F_5}  - \frac{1}{z^3}dz \Bigg] \nonumber \\
     \psi^{\ast}(\tilde{\eta})&=& -\frac{1}{z^3}\Bigg[-\frac{z^3\theta_2}{F_5} - \frac{z^2 \theta_3}{F_5} - \frac{z\theta_4}{F_5} - \frac{\theta_5}{ F_5}  + dz \Bigg] \nonumber \\
     -z^3\psi^{\ast}(\tilde{\eta})&=& \Bigg[z^3\left(\frac{ -\theta_2}{F_5}\right) + z^2 \left(\frac{- \theta_3}{F_5}\right) + z\left(\frac{-\theta_4}{ F_5}\right) + \left(\frac{-\theta_5}{ F_5}\right)  + dz \Bigg] \nonumber \\
     \zeta&:=& -z^3\psi^{\ast}(\tilde{\eta}) = \beta_0 + z \beta_1 + z^2 \beta_2 + z^3 \beta_3 + dz,\nonumber
\end{eqnarray}
where $\displaystyle \beta_0 = -\frac{\theta_5}{F_5(\tau,1)}$, $\displaystyle \beta_1 = -\frac{\theta_4}{F_5(\tau,1)}$, $ \displaystyle \beta_2 = -\frac{\theta_3}{F_5(\tau,1)}$, and $\displaystyle \beta_3 = -\frac{\theta_2}{F_5(\tau,1)}$. 
It is important to observe that $\zeta$ admits a finite G-V-S of length $\geq 3$ provide that $\beta_3 \not\equiv 0$. Consequently, $\f$ also holds this property. Therefore, in this case, we find that either $\f$ is transversely affine, or $\f$ is the pull-back by a rational map of a foliation on $\p^2$, in accordance with Theorem \ref{teogvs}.
\par On the contrary, let us assume that $\beta_3\equiv 0$, implying that $\theta_2\equiv 0$. It follows from (\ref{eq_20}) that
\begin{eqnarray}
    \eta & = & x^{2} \theta_3 + x^3\theta_4 +x^4\theta_5 + x^2 F_5(\tau,1) dx \nonumber\\
 \tilde{\eta} & = & \frac{\eta}{x^2 F_5(\tau,1)} =  \frac{\theta_3}{F_5(\tau,1)} + x\frac{\theta_4}{F_5(\tau,1)} +x^2\frac{\theta_5}{F_5(\tau,1)} + dx\nonumber
\end{eqnarray}
thus
\begin{eqnarray}\label{}
    \tilde{\eta} =\beta_0+x\beta_1+x^2\beta_2+dx
\label{formestrutproj}
\end{eqnarray}
where $\displaystyle \beta_j = \frac{\theta_{j+3}}{F_5(\tau,1)}$ for all $0\leq j\leq 2$. Since $i_{\partial_{x}}(\beta_j)=0$, and $L_{\partial_{x}}(\beta_j)=0$, for all $0\leq j\leq 2$, the integrability condition of $\tilde{\eta}$ implies
\[
\left\{\begin{array}{lll}
    d\beta_0 = \beta_0 \wedge \beta_1 \\
    d\beta_1 = 2\beta_0 \wedge \beta_2\\
    d\beta_2 = \beta_1 \wedge \beta_2
    \end{array}\right.\]
Now, we will proceed to prove the existence of $(\tilde{\eta}, \Omega, \xi)$ that establishes  
 a projective transverse structure for $\tilde{\eta}$, adhering to the condition
$$\left\{\begin{array}{lll}
    d\tilde{\eta} = \tilde{\eta} \wedge \Omega \\
    d\Omega = \tilde{\eta} \wedge \xi\\
    d\xi = \Omega \wedge \xi
    \end{array}\right.$$
By selecting $\Omega = \beta_1 + 2x\beta_2$, we arrive at
\begin{eqnarray}
    \tilde{\eta} \wedge \Omega &  = & (\beta_0 + x\beta_1 + x^2\beta_2 + dx)\wedge(\beta_1 + 2x\beta_2) \nonumber \\
   & = &\beta_0 \wedge \beta_1 + 2x\beta_0\wedge\beta_2+2x^2\beta_1\wedge\beta_2+x^2\beta_2\wedge\beta_1+dx\wedge\beta_1+2xdx\wedge\beta_2
   \nonumber \\
   & = & d\beta_0  +x d\beta_1+2x^2d\beta_2-x^2d\beta_2+dx\wedge\beta_1 +2xdx\wedge\beta_2\nonumber \\
   & = &  d\beta_0  +x d\beta_1+x^2d\beta_2+dx\wedge\beta_1 +2xdx\wedge\beta_2\nonumber \\
   & = & d\tilde{\eta}.
\end{eqnarray}
On the other hand, by computing $d\Omega$, we obtain
\begin{eqnarray}
    d\Omega & = & d\beta_1 +2dx \wedge \beta_2+2xd\beta_2 \nonumber \\
    & = & 2\beta_0 \wedge \beta_2 +2dx \wedge \beta_2 +2x\beta_1 \wedge \beta_2 \nonumber\\
   & = & \beta_0 \wedge (2\beta_2) + dx \wedge (2\beta_2) + x\beta_1 \wedge (2\beta_2) \nonumber \\
   & = & (\beta_0 + dx + x\beta_1 )\wedge (2\beta_2) \nonumber \\
   & = & (\beta_0 + dx + x\beta_1 + x^2\beta_2)\wedge (2\beta_2) \nonumber \\
   & = & \tilde{\eta}\wedge (2\beta_2).
\end{eqnarray}
Now, we choose $\xi= 2\beta_2$, which satisfies the following equality
\begin{eqnarray}
    d\xi &=& 2d\beta_2 = 2\beta_1 \wedge \beta_2 \nonumber \\
    &=& \beta_1\wedge(2\beta_2) \nonumber \\
    &=& (\beta_1+2x\beta_2)\wedge(2\beta_2) \nonumber \\
    &=&\Omega\wedge \xi.
\end{eqnarray}
It is worth that $\Omega$ is closed if, and only if $\beta_2\equiv 0$. This would lead to 
$\xi\equiv 0$, resulting in an affine transverse structure for $\tilde{\eta}$. On the contrary, when $\beta_2 \not\equiv 0$, $\tilde{\eta}$ admits a pure transversely projective structure. Since $\tilde{\f}$ is defined by $\tilde{\eta}$, we can deduce that either  $\f$ is transversely affine, or admits a pure transversely projective structure. 
\end{proof}
\begin{lemma}[Case 4]\label{caso4}
 Suppose that $F_3\equiv 0$, and $F_4 \not\equiv 0$. Then, either $\f$ is transversely affine, or $\f$ is the pull-back by a rational map of a foliation on $\p^2$.
\end{lemma}
\begin{proof}
It follows from (\ref{eq_3}) that
$\eta= x\theta_2 + x^2 \theta_3 + x^3\theta_4 +x^4\theta_5 + (xF_4(\tau,1)+x^2F_5(\tau,1)) dx$.
We will divide the proof into two subcases, addressing each separately. Within each subcase, we will arrive at a 1-form that bears resemblance to
\begin{equation}
    \tilde{\eta} = \beta_0 + x \beta_1 + x^2 \beta_2 + dx
    \label{eq26}
\end{equation}
satisfying $\Le_{\partial_x} (\beta_j) = 0$, and $i_{\partial_x} (\beta_j) = 0$ for all $0\leq j\leq 2$. Having this information at our disposal, we will proceed to derive several of the anticipated results.\\
\textbf{Subcase I.} Suppose that $F_5\equiv 0$. Let $\tilde{\eta}=\frac{\eta}{xF_4}$, then 
\begin{eqnarray}
\tilde{\eta}&=& \frac{\theta_2}{F_4} + x\frac{\theta_3}{F_4} + x^2\frac{\theta_4}{F_4} +x^3\frac{\theta_5}{F_4} + dx\nonumber\\
\tilde{\eta}&=& \beta_0 + x\beta_1 + x^2\beta_2 +x^3\beta_3 + dx \label{eq_case4}
\end{eqnarray}
where $\beta_j = \frac{\theta_{j+2}}{F_4}$, $0\leq j\leq 3$. We can assume that $\beta_3\equiv 0$, because otherwise we would get a finite G-V-S of length $\geq 3$, and in this case, by applying Theorem \ref{teogvs}, either $\f$ is transversely affine, or it is a pull-back by a rational map of a foliation on $\p^2$. \\
\textbf{Subcase II.} Suppose that $F_5 \not\equiv 0$. Let $\tilde{\eta}=\frac{\eta}{x}$, then
$$\tilde{\eta}= \theta_2 + x\theta_3 + x^2 \theta_4 + x^3 \theta_5 + (F_4 + xF_5) dx$$
Consider the birational map $\psi (\tau, z) = \left(\tau,  \frac{\frac{F_4(\tau,1)}{F_5(\tau,1)}z}{1-z}\right)=(\tau,x)$ with inverse $\psi^{-1}(\tau,x)=\left(\tau,\frac{x}{x+\frac{F_4(\tau,1)}{F_5(\tau,1)}}\right)$. A straightforward computation gives  $\psi^{\ast}(\Tilde{\eta})=\frac{F_4^2}{F_5(1-z)^3}\bar{\eta}$, where
$$\bar{\eta}=\beta_0 + z\beta_1 + z^2 \beta_2 + z^3 \beta_3 + dz$$ with 
\begin{eqnarray*}
    \beta_0&=&\frac{F_5 \theta_2}{F_{4}^{2}}\\
\beta_1&=& -3\frac{F_5 \theta_2}{F_{4}^{2}} + \frac{\theta_3}{F_4} + \frac{dF_4}{F_4} - \frac{dF_5}{F_5} \\
\beta_2&=&3\frac{F_5 \theta_2}{F_{4}^{2}} -2\frac{\theta_3}{F_4} + \frac{\theta_4}{F_5} - \frac{dF_4}{F_4} + \frac{dF_5}{F_5} \\
\beta_3&=&-\frac{F_5 \theta_2}{F_{4}^{2}} + \frac{\theta_3}{F_4} - \frac{\theta_4}{F_5} + \frac{F_4 \theta_5}{F_{5}^{2}}
\end{eqnarray*}
As in the previous subcase we can assume that $\beta_3 \equiv 0$ by Theorem \ref{teogvs}.
\par We will now work both cases at the same time as they both fell into the form (\ref{eq_case4}). Remembering that in the initial blow-up we have $\pi^{\ast}(\alpha_2) = x^2[x \theta_2 + F_3(\tau, 1)dx] = x^3 \theta_2$. 
When $\alpha_2 \equiv 0$, we can establish that $\beta_0 \equiv 0$. In this situation, $\eta$ takes the form of $\eta = x\beta_1 + x^2 \beta_2 + dx$, aligning with the conditions of Lemma \ref{caso2}, so let us assume $\alpha_2 \not\equiv 0 $. The integrability of $\omega=\alpha_2 + \alpha_3 + \alpha_4 + \alpha_5$ gives us
\begin{eqnarray}
    0=\omega\wedge d\omega & = & \alpha_2 \wedge d\alpha_2 +\alpha_2 \wedge d\alpha_3 + \alpha_2 \wedge d\alpha_4  + \alpha_2 \wedge d\alpha_5 + \nonumber \\
    & &\alpha_3 \wedge d\alpha_2 + \alpha_3 \wedge d\alpha_3 + \alpha_3 \wedge d\alpha_4 + \alpha_3 \wedge d\alpha_5 +\nonumber \\
    & &\alpha_4 \wedge d\alpha_2 + \alpha_4 \wedge d\alpha_3 + \alpha_4 \wedge d\alpha_4 + \alpha_4 \wedge d\alpha_5+ \nonumber \\
    &  &\alpha_5 \wedge d\alpha_2 + \alpha_5 \wedge d\alpha_3 + \alpha_5 \wedge d\alpha_4 + \alpha_5 \wedge d\alpha_5
\end{eqnarray}
The coefficients of $\alpha_j$ are homogeneous polynomials of degree $j$, then the coefficients of $\alpha_2 \wedge d \alpha_2$ are homogeneous polynomials of degree $3$ and none of the other 2-forms in the other parcels have coefficients of degree 3, then $\alpha_2 \wedge d\alpha_2 =0$. In particular, either $cod(Sing(\alpha_2)) \geq 2$ and $\alpha_2$ define a degree one foliation in $\p^{n-1}$; or $\alpha_2 = h \alpha_1$, where $\alpha_1$ defines a degree zero foliation in $\p^{n-1}$ ($\alpha_2 = h \alpha_1$, in this case $cod(Sing(\alpha_2)) < 2$, and the foliation needs to be saturated and therefore what is left is a 1-form polynomial of degree 1. In both cases $\alpha_2$ has an integrating factor, that is, there exists a function $f$ such that $d(f^{-1}\alpha_2)=0$.
\par In \textbf{Subcase I}, we have
\begin{align}
    \pi^{\ast}\Bigg(\frac{\alpha_2}{f}\Bigg) = \frac{\theta_2}{f(\tau,1)} \Rightarrow d \Bigg(\frac{\theta_2}{f(\tau,1)} \Bigg) = 0 \Rightarrow d \Bigg(\frac{F_4 (\tau,1)}{f(\tau,1)}\beta_0 \Bigg) = 0
    \label{eq29}
\end{align}
We put $\displaystyle F_1(\tau):= \frac{f(\tau,1)}{F_4(\tau,1)}$ for this case.
\par In \textbf{Subcase II}, we have
\begin{align}
    \pi^{\ast}\Bigg(\frac{\alpha_2}{f}\Bigg) = \frac{\theta_2}{f(\tau,1)} \Rightarrow d \Bigg(\frac{\theta_2}{f(\tau,1)} \Bigg) = 0 \Rightarrow d \Bigg(\frac{F_{4}^{2} (\tau,1)}{F_5 (\tau,1) f(\tau,1)}\beta_0 \Bigg) = 0
    \label{eq30}
\end{align}
We put $\displaystyle F_2(\tau):= \frac{F_5 (\tau,1) f(\tau,1)}{F_{4}^{2}(\tau,1)}$ for this case.
Now, let us consider the birational map $\Phi_i(\tau,z) = (\tau, F_i(\tau)z)=(\tau, x)$, $i=1,2$. If $\tilde{\eta}$ is the one that appears in \textbf{Subcase I}, just use $\Phi_1$, if it is $\tilde{\eta}$ in \textbf{Subcase II}, just use $\Phi_2$, so we will omit the indices $i$ of $\Phi_i$ and $F_i$ because the calculation we will perform using these transformations in one case is identical to the other. If $\tilde{\eta}$ is like in $(\ref{eq_case4})$, a straightforward calculation gives $\Phi^{*}(\tilde{\eta})=F\bar{\eta}$, where 
\[\bar{\eta}=dz+\tilde{\beta}_0+z\tilde{\beta}_1+z^2\tilde{\beta}_2\] with 
$\tilde{\beta}_0=\frac{\beta_0}{F}$, $\tilde{\beta}_1=\beta_1+\frac{dF}{F}$, and $\tilde{\beta}_2=F\beta_2$. 
Let us consider the birational map $\phi(\tau, w) = (\tau, \frac{1}{w}) = (\tau, z)$. Then
$\phi^{\ast}(\bar{\eta})=-w^2 \hat{\eta}$,
where $\hat{\eta} = dw - \Tilde{\beta_2} - w\Tilde{\beta_1} - w^2 \Tilde{\beta_0}$. 
Since $i_{\partial_{w}}(\tilde{\beta}_j)=0$ and $L_{\partial_{w}}(\tilde{\beta}_j)=0$, for all $0\leq j\leq 2$, the integrability of $\hat{\eta}$ implies
\begin{equation}
\left\{\begin{array}{lll}
    d\Tilde{\beta_0} = \Tilde{\beta_0} \wedge \Tilde{\beta_1} \\
    d\Tilde{\beta_1} = 2\Tilde{\beta_0} \wedge \Tilde{\beta_2}\\
    d\Tilde{\beta_2} = \Tilde{\beta_1} \wedge \Tilde{\beta_2}
    \end{array}\right.
    \label{eq31}
\end{equation}
From (\ref{eq29}) and (\ref{eq30}), we get $d\Tilde{\beta_0} =0$, and by the first equation in (\ref{eq31}) we obtain $\Tilde{\beta_0} \wedge \Tilde{\beta_1} =0$. Denoting by $\mathcal{M}_{k}$ the set of meromorphic functions in $\p^{k}$. It follows from $\Tilde{\beta_0} \wedge \Tilde{\beta_1} =0$ that there exists $g \in \mathcal{M}_{n-1}$ such that $\Tilde{\beta_1} = g\Tilde{\beta_0}$. The second relation in (\ref{eq31}) gives us
\begin{align}
    d\Tilde{\beta_1} = dg \wedge \Tilde{\beta_0} = 2 \Tilde{\beta_0} \wedge \Tilde{\beta_2} & \Rightarrow dg \wedge \Tilde{\beta_0} + 2 \Tilde{\beta_2} \wedge \Tilde{\beta_0}=0 \nonumber \\
    & \Rightarrow (dg + 2\Tilde{\beta_2})\wedge \Tilde{\beta_0} =0. \nonumber
\end{align}
Therefore, there exists $h \in \mathcal{M}_{n-1}$ such that $$\displaystyle \frac{dg}{2} + \Tilde{\beta_2} = h \Tilde{\beta_0} \Rightarrow \Tilde{\beta_2} = h \Tilde{\beta_0} - \frac{dg}{2}.$$ The third relation in (\ref{eq31}) implies
\begin{align}
    d \Tilde{\beta_2} = dh \wedge \Tilde{\beta_0} = & \Tilde{\beta_1} \wedge \Tilde{\beta_2} = g\Tilde{\beta_0} \wedge \Bigg(h \Tilde{\beta_0} -  \frac{dg}{2}\Bigg) = g \frac{dg}{2} \wedge \Tilde{\beta_0} \nonumber \\
    \Rightarrow & dh \wedge \Tilde{\beta_0} = g \frac{dg}{2} \wedge \Tilde{\beta_0} \Rightarrow \Bigg(dh - g \frac{dg}{2} \Bigg) \wedge \Tilde{\beta_0} = 0 \nonumber \\
    \Rightarrow & d \Bigg(h - \frac{g^2}{4} \Bigg) \wedge \Tilde{\beta_0} = 0.
\label{eq32}
\end{align}
We will designate $\mathcal{G}$ as the foliation generated by $\Tilde{\beta_0}$ in $\p^{n-1}$. Notably, $\Tilde{\beta_0}$ gives rise to a foliation in $\p^{n-1}$, due to its definition within $\C^{n-1}$, and additionally, it is closed. 

We have two possibilities:\\
\textbf{Possibility I.} $\mathcal{G}$ has no non-constant first integral. We claim that $\omega$ has an integral factor. In fact, by (\ref{eq32}) $$ d \left(h - \frac{g^2}{4} \right) \wedge \Tilde{\beta_0} = 0 \Rightarrow h= \frac{g^2} {4} +c,$$ where $c \in \C$, otherwise $\mathcal{G}$ would have non-constant first integral. We can then write
\begin{eqnarray}
    \hat{\eta} &=& dw - \Tilde{\beta_2} - w\Tilde{\beta_1} - w^2 \Tilde{\beta_0} = dw - \left(\frac{g^2}{4} + c\right)\Tilde{\beta_0} - \frac{dg}{2} - g\Tilde{\beta_0}w - \Tilde{\beta_0}w^2 \nonumber \\
    &=& dw - \frac{dg}{2} - \left(\frac{g}{4} +c + gw +w^2\right)\Tilde{\beta_0} \nonumber \\
    &=& d\left(w - \frac{g}{2}\right) - \left(\left[w + \frac{g}{2}\right]^2 +c\right)\Tilde{\beta_0},
\end{eqnarray}
in particular, if we set $\mathcal{P} := \left(\left[w + \frac{g}{2}\right]^2 +c\right)^{-1}\hat{\eta}$, then
$$\mathcal{P} := \frac{d(w - \frac{g}{2})}{\left[w + \frac{g}{2}\right]^2 +c} - \Tilde{\beta_0} \Rightarrow  d\Tilde{\beta_0} =0.$$
Therefore $\hat{\eta}$ has an integral factor and thus $\omega$ as well. Then there exists $h$ such that
\begin{align}
    d(h\omega) = 0 &\Rightarrow dh \wedge \omega + h d\omega = 0 \nonumber \\
    & \Rightarrow -\frac{dh}{h} \wedge \omega = dw.
\end{align}
Consequently, $\f$ is transversely affine.\\
\textbf{Possibility II.} $\mathcal{G}$ has non-constant first integral. We assert that $\f$ is the pull-back of a Riccati equation on $\p^1 \times \p^1$ by a birational map.
Indeed, by Stein's Factorization Theorem $\mathcal{G}$ has a meromorphic first integral, say $f$, with connected fibers: if $\phi \in \mathcal{M}_{n-1}$ and $d\phi \wedge df =0$ then there exists $\psi \in \mathcal{M}_{1}$ such that $\phi = \psi(f)$ where $\psi (f) = \psi \circ f $. \\
On the other hand, the relation (\ref{eq32}) implies that there exists $\phi_2 \in \mathcal{M}_{1}$ such that 
$$d\left(h - \frac{g^2}{4} \right)\wedge \Tilde{\beta_0} = 0 \Rightarrow d\left(h - \frac{g^2}{4} \right)\wedge\phi_1 df =0 \Rightarrow d\left(h - \frac{g^2}{4} \right)\wedge df =0$$
and by Stein's Factorization Theorem 
$$h - \frac{g^2}{4} = \psi_2 (f) \Rightarrow h = \psi_2 (f) + \frac{g^2}{4},$$ replacing in $\hat{\eta}$ we have
\begin{eqnarray}
    \hat{\eta} &= & dw - \left(h \Tilde{\beta_0} - \frac{dg}{2}\right) - w (g\Tilde{\beta_0}) - \Tilde{\beta_0}w^2 \nonumber \\
    & = & d \left(w - \frac{g}{2}\right) + (-h -wg -w^2) \Tilde{\beta_0} \nonumber \\
    & = & d \left(w - \frac{g}{2}\right) - \left(\frac{g^2}{4} + \psi_2 (f) + wg + w^2\right) \psi_1 (f) df \nonumber \\
    & = & d \left(w - \frac{g}{2}\right) - \left(\left[w + \frac{g}{2}\right]^2  + \psi_2 (f)\right) \psi_1 (f) df
\end{eqnarray}
Consider the rational map $\Phi_1 : \p^{n-1} \times \p^1 \dashrightarrow \p^1 \times \p^1$ given by $$\Phi_1 (\tau, w) = \left( f(\tau), w - \frac{g(\tau)}{2}\right)=: (x,y)$$ Then $\hat{\eta} = \Phi^{\ast}(\theta)$, where
$\theta = dy - (y^2 + \psi_2 (x))\psi_1 (x) dx$.
Note that $\theta$ is an integrable 1-form that defines a Riccati equation in $\p^1 \times \p^1$ and furthermore, it has a transversely affine structure as we wanted, then $\f$ is transversely affine.
\end{proof}

\begin{lemma}[Case 5]
 Suppose that $F_5\equiv 0$, and $F_3 \not\equiv 0 \not\equiv F_4$. Then, there exists a birational map $\Psi:\mathbb{P}^{n-1}\times\mathbb{P}^{1}\dashrightarrow \mathbb{P}^{n}$ such that the foliation $\Psi^{*}(\f)$ is defined by a 1-form described as follows:
\[\eta_t = \beta_0 + t\beta_1 + t^2 \beta_2  + t^3 \beta_3 + t^4 \beta_4 -tdt,\]
where the 1-forms $\beta_j$ do not depend on $t\in\mathbb{P}^{1}$, for all $0\leq j\leq 4$.
\label{caso5}
\end{lemma}
\begin{proof}
We consider $(\tau,x)\in\C^{n-1}\times\mathbb{C}\subset\mathbb{P}^{n-1}\times\mathbb{P}^{1}$ as outlined in (\ref{explosion}), so that the blowing-up $\sigma$ extends to a birational map $\psi_1:\mathbb{P}^{n-1}\times\mathbb{P}^{1}\dashrightarrow \mathbb{P}^{n-1}\times\mathbb{P}^{1}$. If $F_5\equiv 0$,
it can be inferred from (\ref{eq_3}) that
    \begin{equation}\label{case5_eta}
    \eta = x\theta_2 + x^2\theta_3 + x^3\theta_4 + x^4\theta_5 + (F_3(\tau,1) + xF_4(\tau,1))dx, 
    \end{equation}
    and the pull-back foliation $\psi_1^{*}(\f)$ is induced by $\eta$. 
    We will perform a series of pull-backs by birational maps on $\eta$ until the desired result is achieved. To begin, we initiate the pull-back by applying the birational map $\psi_z:\mathbb{P}^{n-1}\times\mathbb{P}^{1}\dashrightarrow \mathbb{P}^{n-1}\times\mathbb{P}^{1}$ defined as 
     $\psi_z (\tau, z) =  (\tau, \frac{1}{z}) = (\tau, x)$. Consequently, $\psi_{z}^{\ast} (\eta)=\frac{\eta_z}{z^4}$, where
\[\eta_z = \theta_5 + z \theta_4 + z^2 \theta_3 + z^3 \theta_2 - (zF_4(\tau,1) + z^2 F_3(\tau,1))dz.\]
Continuing, we consider the birational map $\psi_t:\mathbb{P}^{n-1}\times\mathbb{P}^{1}\dashrightarrow \mathbb{P}^{n-1}\times\mathbb{P}^{1}$ defined as 
$\displaystyle \psi_t (\tau, t) = \left(\tau,\frac{\frac{F_4(
\tau,1)}{F_3(\tau,1)}t}{(1-t)}\right)=(\tau,z)$. A direct calculation yields 
$\psi^{*}_t(\eta_z)=\frac{F_4^{3}}{(1-t)^4F_3^2}\eta_t$, where
\[\eta_t = \beta_0 + t\beta_1 + t^2 \beta_2  + t^3 \beta_3 + t^4 \beta_4 -tdt\]
with 
\begin{eqnarray*}
\beta_0 &=& \frac{F_{3}^{2}}{F_{4}^{3}}\theta_5\\ 
\beta_1 &=& -4\frac{F_{3}^{2}}{F_{4}^{3}}\theta_5 +\frac{F_3}{F_{4}^{2}} \theta_4 \\
\beta_2 &=& 6\frac{F_{3}^{2}}{F_{4}^{3}}\theta_5 -3 \frac{F_3}{F_{4}^{2}} \theta_4 + \frac{\theta_3}{F_{4}} - \frac{dF_4}{F_{4}} + \frac{dF_3}{F_{3}} \\
\beta_3 &=& -4\frac{F_{3}^{2}}{F_{4}^{3}}\theta_5 +3 \frac{F_3}{F_{4}^{2}} \theta_4 -2 \frac{\theta_3}{F_{4}} +2 \frac{dF_4}{F_{4}} + \frac{dF_3}{F_{3}} + \frac{\theta_2}{F_{3}}\\
\beta_4 &=& \frac{F_{3}^{2}}{F_{4}^{3}}\theta_5 - \frac{F_3}{F_{4}^{2}} \theta_4 + \frac{\theta_3}{F_{4}} - \frac{\theta_2}{F_{3}}
\end{eqnarray*}
The proof concludes by observing that the 1-forms $\beta_j$ are dependent solely on $\tau$, for all $0\leq j\leq 4$, and taking $\Psi:=\psi_1\circ\psi_z\circ\psi_t$.
\end{proof}

\begin{lemma}[Cases 6 and 7]\label{caso6}
 Suppose that $F_4\equiv 0$, and $F_3 \not\equiv 0$, $F_5\not\equiv 0$; or $F_3 \not\equiv 0$, $F_4 \not\equiv 0$, and $F_5 \not\equiv 0$. Then, the foliation $\f$ is either transversely affine or is the pull-back by a rational map of a foliation on $\p^2$, or
 there exists a birational map $\Psi:\mathbb{P}^{n-1}\times\mathbb{P}^{1}\dashrightarrow \mathbb{P}^{n}$ such that the foliation $\Psi^{*}(\f)$ is defined by a 1-form described as follows:
\[\eta_t = \beta_0 + t\beta_1 + t^2 \beta_2  + t^3 \beta_3 + t^4 \beta_4 -tdt,\]
where the 1-forms $\beta_j$ do not depend on $t\in\mathbb{P}^{1}$, for all $0\leq j\leq 4$. 
\end{lemma}
\begin{proof}
As in Lemma \ref{caso5}. the map $\sigma$ extends to a birational map $\psi_1:\mathbb{P}^{n-1}\times\mathbb{P}^{1}\dashrightarrow \mathbb{P}^{n-1}\times\mathbb{P}^{1}$. Then, following from equation (\ref{eq_3}), the foliation $\psi^{*}(\f)$ is defined by the 1-form: 
\begin{align}
    \eta = x\theta_2 + x^2\theta_3 + x^3\theta_4 + x^4\theta_5 + (F_3(\tau,1) + x F_4(\tau,1) + x^2 F_5(\tau,1))dx.
\end{align}
Since $F_5 \not\equiv 0$, we can divide the expression of $\eta$ by $F_5(\tau,1)$ and obtain:
\begin{align}
    \eta_1 = x\alpha_2 + x^2\alpha_3 + x^3\alpha_4 + x^4\alpha_5 + \left(\frac{F_3(\tau,1)}{F_5(\tau,1)} + x \frac{F_4(\tau,1)}{F_5(\tau,1)} + x^2\right)dx, 
\end{align}
where $\alpha_j=\theta_j/F_{5}(\tau,1)$, for all $2\leq j\leq 5$.
Now, we factorize the polynomial 
\[x^2+x \frac{F_4(\tau,1)}{F_5(\tau,1)}+\frac{F_3(\tau,1)}{F_5(\tau,1)}=(x-c_1(\tau))(x-c_2(\tau)).\] 
Note that $c_1(\tau)$ and $c_2(\tau)$ are not identically zero, as assumed from the hypotheses $F_3\not\equiv 0$, $F_5 \not\equiv 0$, and $c_1(\tau)\cdot c_2(\tau)=\frac{F_3(\tau,1)}{F_5(\tau,1)}$.
\par First, we consider the birational map $\psi_z:\mathbb{P}^{n-1}\times\mathbb{P}^{1}\dashrightarrow \mathbb{P}^{n-1}\times\mathbb{P}^{1}$ defined by
     $\psi_z (\tau, z) =\left(\tau,\frac{c_1(\tau) z}{z-1}\right)=(\tau,x)$. A direct calculation yields $\psi_z^{*}(\eta_1)=\frac{c_1^2(\tau)}{(z-1)^4}\eta_z,$ where 
     \begin{equation}\label{case6_7}
     \eta_z=z\gamma_1+z^2\gamma_2+z^3\gamma_3+z^4\gamma_4-[(c_1(\tau)-c_2(\tau))z+c_2(\tau)]dz
     \end{equation}
  with 
  \begin{eqnarray*}
  \gamma_1&=&\frac{\alpha_2}{c_1(\tau)}-c_2(\tau)\frac{dc_1(\tau)}{c_1(\tau)}\\
  \gamma_2&=&\frac{3\alpha_2}{c_1(\tau)}+\alpha_3+(2c_2(\tau)-c_1(\tau))\frac{dc_1(\tau)}{c_1(\tau)}\\
  \gamma_3&=&\frac{-3\alpha_2}{c_1(\tau)}-2\alpha_3-c_1(\tau)\alpha_4+(c_1(\tau)-c_2(\tau))\frac{dc_1(\tau)}{c_1(\tau)}\\
  \gamma_4&=&\frac{\alpha_2}{c_1(\tau)}+\alpha_3+c_1(\tau)\alpha_4+c_1^{2}(\tau)\alpha_5
  \end{eqnarray*}   
  Now we will analyze the expression of $\eta_z$ in (\ref{case6_7}), observing that we have the following subcases:\\
  \textbf{Subcase I.} $c_1(\tau)=c_2(\tau)$. In this situation, we have that 
\begin{equation}
     \eta_z=z\gamma_1+z^2\gamma_2+z^3\gamma_3+z^4\gamma_4+c_2(\tau)dz
     \end{equation}
 Since $c_2(\tau)$ is not identically zero, we can divide $\eta_z$ by $c_2(\tau)$ and obtain a 1-form $\tilde{\eta}$ that is equivalent to the 1-form derived in equation (\ref{eq27}) of Lemma \ref{caso2}. Hence, we can conclude that the foliation $\f$ is either transversely affine or is the pull-back by a rational map of a foliation on $\p^2$.\\
   \textbf{Subcase II.} $c_1(\tau)\neq c_2(\tau)$.
 In this subcase, the 1-form $\eta_z$ is equivalent to the 1-form derived in equation (\ref{case5_eta}) of Lemma \ref{caso5}. Therefore, we can conclude that 
 there exists a birational map $\Psi:\mathbb{P}^{n-1}\times\mathbb{P}^{1}\dashrightarrow \mathbb{P}^{n-1}\times\mathbb{P}^{1}$ such that the foliation $\Psi^{*}(\f)$ is defined by a 1-form described as follows:
\[\eta_t = \beta_0 + t\beta_1 + t^2 \beta_2  + t^3 \beta_3 + t^4 \beta_4 -tdt,\]
where the 1-forms $\beta_j$ do not depend on $t\in\mathbb{P}^{1}$, for all $0\leq j\leq 4$.
 \end{proof}
\par We can condense the results obtained in the above lemmas in the following proposition:
\begin{proposition}
In the above situation, we have the five possibilities:
\begin{itemize}
    \item[(i)] either $\f$ is transversely affine outside a compact hypersurface;
    \item[(ii)] or $\f$ is pure transversely projective outside a compact hypersurface;
    \item[(iii)] or $\f$ is a pull-back by a rational map of a foliation on $\p^2$;
    \item[(iv)] or $\f$ is a pull-back by a linear map $\pi : \p^n \to \p^{n-1}$ of a foliation of degree four on $\p^{n-1}$.
    \item[(v)] or there exists a birational map $\Psi:\mathbb{P}^{n-1}\times\mathbb{P}^{1}\dashrightarrow \mathbb{P}^{n}$ such that the foliation $\Psi^{*}(\f)$ is defined by a 1-form described as follows:
\[\eta_t = \beta_0 + t\beta_1 + t^2 \beta_2  + t^3 \beta_3 + t^4 \beta_4 -tdt,\]
where the 1-forms $\beta_j$ do not depend on $t\in\mathbb{P}^{1}$, for all $0\leq j\leq 4$.
\end{itemize}
In particular, if $n=3$ then $\f$ satisfies $(i)$, $(ii)$, $(v)$, or $(iii)$.
\label{aprop}
\end{proposition}
\subsection{End of the proof of Theorem \ref{teoprincip}}
We give the proof by induction on the dimension $n\geq 3$. If $n=3$, then Theorem \ref{teoprincip} follows from Corollary \ref{cor1} and Proposition \ref{aprop}. Let us assume that Theorem \ref{teoprincip} is true for $n-1\geq 3$ and prove that it holds for $n$.
\par Let $\f$ be a codimension one foliation of degree four on $\p^n$, $n\geq 4$. It follows from Corollary \ref{cor1} and Proposition \ref{aprop} that, either $\f$ satisfies one of the conclusions of Theorem \ref{teoprincip}, or $\f$ is the pull-back by a linear map $\pi : \p^n \to \p^{n-1}$ of a foliation $\f_{n-1}$ of degree four on $\p^{n-1}$. In this last case, as Theorem \ref{main_theo} holds true for $n-1$, it follows that one the five possibilities outlined below must also be true:
\begin{enumerate}
    \item[(i)] $\f_{n-1}$ has a rational first integral, say $F:\p^{n-1}\dashrightarrow\p^1$. In this case, $F\circ\pi$ is a rational first integral of $\f$ and we are done.
    \item[(ii)] $\f_{n-1}$ is transversely affine. In this case, $\f_{n-1}$ admits a G-V-S of length one. Hence, $\f$ also admits a G-V-S of length one by Remark \ref{pull-back-gvs}.  \item[(iii)] $\f_{n-1}$ is transversely projective. In this case, $\f_{n-1}$ admits a G-V-S of length two. Hence, $\f$ also admits a G-V-S of length two by Remark \ref{pull-back-gvs}.
    \item[(iv)] $\f_{n-1}=\Phi^{*}(\mathcal{G})$, where $\mathcal{G}$ is a foliation on $\p^2$ and $\Phi:\p^{n-1}\dashrightarrow\p^2$ a rational map. In this case, we get $\f=(\Phi\circ\pi)^{*}(\mathcal{G})$ and we are done.  
    \item[(v)] There exists a birational map $\Psi_1:\mathbb{P}^{n-2}\times\mathbb{P}^{1}\dashrightarrow \mathbb{P}^{n-1}$ such that the foliation $\Psi_1^{*}(\f_{n-1})$ is defined by a 1-form described as follows:
\[\eta_t = \beta_0 + t\beta_1 + t^2 \beta_2  + t^3 \beta_3 + t^4 \beta_4 -tdt,\]
where the 1-forms $\beta_j$ do not depend on $t\in\mathbb{P}^{1}$, for all $0\leq j\leq 4$. Now consider any rational map $\phi:\p^{n-1}\times\p^{1}\dashrightarrow\p^{n-2}\times\p^1$ such that it fixes the variable at $\p^1$, and choose any birational map $\Psi$ such that $\pi\circ\Psi=\Psi_1\circ\phi$. With this, we have that 
\[\Psi^{*}(\f)=\Psi^{*}(\pi^{*}(\f_{n-1}))=(\pi\circ\Psi)^{*}(\f_{n-1})=(\Psi_1\circ\phi)^{*}(\f_{n-1})=\phi^{*}(\Psi_1^{*}(\f_{n-1}))\]
Using the fact that $\phi$ fixes the variable over $\p^1$, we deduce that $\Psi^{*}(\f)$ is defined by a 1-form similar to item $(v)$. This concludes the proof of Theorem \ref{teoprincip}.
\end{enumerate}

\section{Proof of Theorem \ref{teorema_B}}\label{prova_B}
\par Let $\f$ be a codimension one holomorphic foliation of degree $d\geq 4$ in $\p^n$, $n \geq 3$. 
Suppose that one of the two conditions is satisfied: 
 \begin{enumerate}
 \item for all $p\in\sing(\f)$, we have $\mathcal{J}(\f,p)=1$;
 \item there exists $p\in\sing(\f)$ such that $\mathcal{J}(\f,p)\geq d-1$.
 \end{enumerate}
 In the first case, $\f$ admits a rational first integral by invoking Corollary \ref{cor1}. This consequently establishes the validity of assertion $(i)$ within Theorem \ref{teorema_B}.
\par Therefore, we shall assume that there exists a point $p \in \p^n$ such that $\mathcal{J}(\f,p) \geq  d-1$. 
By employing affine coordinates $(z_1,\ldots,z_n)\in \C^n \subset \p^n$, where $p=0 \in \C^n$, we can conveniently consider $\f|_{\C^n}:\omega=0$, where $\omega$ is a polynomial 1-form in $\C^n$ expressed as follows: 
\begin{equation}
\omega = \alpha_{d-1} + \alpha_{d} + \alpha_{d+1},
\end{equation} 
here, $\alpha_j$ corresponds to  homogeneous polynomial 1-forms of degree $j$, $d-1\leq j\leq d+1$, and 
\begin{eqnarray}\label{eq_R}
i_R (\alpha_{d+1})=0,\quad \quad\text{with}\quad R = \sum_{i=1}^{n} z_i \partial z_i.
\end{eqnarray}
Once again, we will express $\alpha_j$ as: $\displaystyle \alpha_j (z) := \sum_{i=1}^{n} P_{ji}(z)dz_i$, with $d-1\leq j\leq d+1$. Additionally, we  introduce
\[F_j (z) :=i_{R}(\alpha_{j-1})=\displaystyle \sum_{i=1}^{n}z_{i}\cdot P_{j-1 i}(z),\quad d-1\leq j\leq d+1,\] where $P_{j-1 i}$ are homogeneous polynomials of degree $j-1$. Note that $F_{d+2} \equiv 0$ by (\ref{eq_R}). 
\par We proceed to examine the pull-back of 
$\omega$ through the process of  blowing-up of $\p^n$ at $0 \in \C^n \subset \p^n$. Let $\sigma : \tilde{\p}^n \rightarrow \p^n$ denote the blow-up at $0 \in \C^n\subset\p^n$, and let $\tilde{\f}$ represent the strict transform of $\f$ by $\sigma$. Our objective is to calculate $\sigma^{\ast}( \omega)$ within the chart
\begin{align}\label{explosion}
(\tau_1, \ldots, \tau_{n-1},x)=(\tau,x) \in \C^{n-1}\times \C \mapsto (x \tau,x) = (z_1, \ldots,z_n) \in \C^n \subset \p^n. 
\end{align}
 We have 
\begin{eqnarray}
\sigma^{\ast}(\omega) &= & x^{d-1}(x\theta_{d-1}+x^2\theta_{d}+x^3\theta_{d+1}+(F_d(\tau,1)+x F_{d+1}(\tau,1)+x^2 F_{d+2}(\tau,1)))\nonumber
\end{eqnarray}
where $$\displaystyle \theta_j = \sum_{i=1}^{n-1}P_{ji}(\tau,1)d\tau_i,\quad\quad d-1\leq j\leq d+1$$ depends only on $\tau$. 
Utilizing the condition $F_{d+2}(\tau,1)\equiv 0$, we derive the 1-form $\eta$ as follows: 
\begin{equation}\label{eq_3}
\eta=x\theta_{d-1}+x^2\theta_{d}+x^3\theta_{d+1}+(F_d(\tau,1)+x F_{d+1}(\tau,1))dx.   
\end{equation}
 This 1-form serves to define the foliation $\tilde{\f}$ in the chart $(\tau,x)$.
\par Given the aforementioned conditions, we are presented with the subsequent possibilities for $F_i$:
\begin{itemize}
    \item[(1)] $F_d \equiv 0$, and $F_{d+1} \not \equiv 0$;
    \item[(2)] $F_d \not \equiv 0$, and $F_{d+1} \equiv 0$;
    \item[(3)] $F_d \not \equiv 0$, and $F_{d+1} \not \equiv 0$.
\end{itemize}
\par In the case (1), after dividing $\eta$ by $x\cdot F_{d+1}$, we have
\[\eta_1=\gamma_0+x\gamma_1+x^2\gamma_2+dx,\] where 
$\gamma_0=\theta_{d-1}/F_{d+1}(\tau,1)$, $\gamma_1=\theta_{d}/F_{d+1}(\tau,1)$, and $\gamma_{2}=\theta_{d+1}/F_{d+1}(\tau,1)$ depends only on $\tau$.
The 1-form $\eta_1$ is equivalent to the 1-form from (\ref{formestrutproj}) of Lemma \ref{caso3}. Therefore, we can conclude either $\f$ is transversely affine, or $\f$ is the pull-back by a rational map of a foliation on $\p^2$, or $\f$ is pure transversely projective.
\par In the case (2), after dividing $\eta$ by $F_d$, we have
\[\eta_1=x\gamma_1+x^2\gamma_2+x^3\gamma_3+dx,\] where 
$\gamma_1=\theta_{d-1}/F_{d}(\tau,1)$, $\gamma_2=\theta_{d}/F_{d}(\tau,1)$, and $\gamma_{3}=\theta_{d+1}/F_{d}(\tau,1)$ depends only on $\tau$. The 1-form $\eta_1$ is equivalent to the 1-form from (\ref{eq27}) of Lemma \ref{caso2}, Subcase II. Consequently, we can deduce that either  $\f$ is transversely affine, or $\f$ is the pull-back by a rational map of a foliation on $\p^2$. 
\par In the case (3), we consider the birational map $\psi_z:\p^{n-1}\times\p^{1}\dashrightarrow\p^{n-1}\times\p^1$ defined as $\psi_z(\tau,z)=\left(\tau,\frac{F_d(\tau,1)z}{1-F_{d+1}(\tau,1)z}\right)$. A direct calculation yields $\psi_z^{*}(\eta_1)=\frac{F_{d}^2(\tau,1)}{(1-F_{d+1}(\tau,1)z)}\eta_z$, where 
\[\eta_z = z\beta_1 + z^2\beta_2 + z^3 \beta_3  + dz\]
with 
\begin{eqnarray*}
\beta_1 &=& \frac{\theta_{d-1}}{F_{d}}+\frac{dF_{d}}{F_d}\\ 
\beta_2 &=& -\frac{F_{d+1}}{F_{d}}\theta_{d-1} +\theta_{d}- F_{d-1}\frac{dF_{d}}{F_{d}} +dF_{d+1} \\
\beta_3 &=& -F_{d+1}\theta_{d}+F_{d}\theta_{d+1} 
\end{eqnarray*}
Once again, the 1-form $\eta_z$ is equivalent to the 1-form from (\ref{eq27}) of Lemma \ref{caso2}, Subcase II. Thus, we can deduce that either  $\f$ is transversely affine, or $\f$ is the pull-back by a rational map of a foliation on $\p^2$. This finishes the proof of Theorem \ref{teorema_B}.

\vspace{2cm}

\textit{Acknowledgement.} We would like to express our gratitude to Miguel Rodr\'iguez Pe\~na, Maur\'icio Corr\^ea, Omegar Calvo Andrade, Ruben Lizarbe, and Jorge Vit\'orio Pereira for their valuable comments and references that  greatly contributed to enhancing the quality of the paper's results. This paper represents a compilation of the outcomes from the second author's doctoral thesis at the Federal University of Minas Gerais, Brazil.

\end{document}